\documentclass[final, twoside,a4paper]{amsart}
\usepackage[latin1]{inputenc}
\usepackage[T1]{fontenc}
\usepackage{textcomp}
\usepackage[english]{babel}
\usepackage{amsmath, amsfonts,amssymb,amsopn,amscd,amsthm}
\usepackage{mathrsfs}
\usepackage{dsfont}
\usepackage{graphicx}
\usepackage{enumerate}
\usepackage[usenames,dvipsnames]{xcolor}
	\usepackage{tikz}
 \usepackage{hyperref}
 \hypersetup{colorlinks=true,linkcolor=blue,citecolor=blue,linktoc=page}

\def\P{\mathds{P}}
\def\E{\mathds{E}} %

\renewcommand{\d}{\mathrm{d}}

\newcommand{\R}{\mathds{R}}                     

\newcommand{\abs}[1]{\left\vert#1\right\vert}

\newcommand{\mbf}[1]{\mathbf #1}
\newcommand{\eq}{\begin{equation}}
\newcommand{\qe}{\end{equation}}
\newcommand{\G}{\mbf \Gamma}

\newcommand{\Lambert}[1]{\mathds W_{#1}}

\def\ni{\noindent}
\def \UN{\mathds 1}

\newtheoremstyle{theo}
{}
{}
{\itshape}
{\parindent}
{\bf}
{\ ---}
{.5em}
{}%
\theoremstyle{theo}
\newtheorem{theorem}{Theorem}
\newtheorem{lemma}{Lemma}

\newtheorem{exemple}{Exemple}

\newtheoremstyle{def}%
{}
{}
{\itshape}
{\parindent}
{\bf}
{\ ---}
{.5em}
{}%
\theoremstyle{def}
\newtheorem*{definition}{Definition}

\theoremstyle{remark}
\newtheorem*{remark}{Remark}

\title[A Rice method proof of NSP]{A Rice method proof of the Null-Space Property over the Grassmannian}
\date{\today}
\keywords{Rice Method; High-dimensional statistics; $\ell_{1}$-minimization; Null-Space Property; Random processes theory; }
\subjclass[2010]{62J05; 62H12; 62F12; 62E20;} 
\author{J.-M. Aza\"is}\address{JMA is with the Institut de Math\'ematiques de Toulouse (CNRS UMR 5219). Universit\'e Paul Sabatier, 118 route de Narbonne, 31062 Toulouse, France.}
\email{jean-marc.azais@math.univ-toulouse.fr}
\author{Y. De Castro}\address{YDC is with the Laboratoire de Math\'ematiques d'Orsay, Univ. Paris-Sud, CNRS, Universit\'e Paris-Saclay, 91405 Orsay, France.}
\email{yohann.decastro@math.u-psud.fr}
\author{S. Mourareau}\address{SM is with the Institut de Math\'ematiques de Toulouse (CNRS UMR 5219). Universit\'e Paul Sabatier, 118 route de Narbonne, 31062 Toulouse, France.}
\email{stephane.mourareau@math.univ-toulouse.fr}
\begin{document}

\begin{abstract}
The Null-Space Property (NSP) is a necessary and sufficient condition for the recovery of the largest coefficients of solutions to an under-determined system of linear equations. Interestingly, this property governs also the success and the failure of recent developments in high-dimensional statistics, signal processing, error-correcting codes and the theory of polytopes. 

Although this property is the keystone of $\ell_{1}$-minimization techniques, it is an open problem to derive a closed form for the phase transition on NSP. In this article, we provide the first proof of NSP using random processes theory and the Rice method. As a matter of fact, our analysis gives non-asymptotic bounds for NSP with respect to unitarily invariant distributions. Furthermore, we derive a simple sufficient condition for NSP. 
\end{abstract}
\maketitle
%

\section{Introduction}
\subsection{Null-Space Property}
One of the simplest inverse problem can be described as follows: given a matrix $X\in\R^{n\times p}$ and $y\in\mathrm{Im}(X)$, can we faithfully recover $\beta^{\star}$ such that the identity $y=X\beta^{\star}$ holds? In the ideal case where $n\geq p$ and the matrix $X$ is one to one (namely, the model is identifiable), this problem is elementary. However, in view of recent applications in genetics, signal processing, or medical imaging, the frame of high-dimensional statistics is governed by the opposite situation where $n< p$. To bypass the limitations due to the lack of identifiability, one usually assumes that the matrix $X$ is at random and one considers the $\ell_{1}$-minimization procedure \cite{MR1639094}:
\eq\label{prog:BP}
\tag{$P_{\ell_1}$}
\Delta_{X}(\beta^{\star})\in\arg\min_{X\beta=X\beta^{\star}}\lVert \beta\lVert_{1}\,,
\qe
where $\beta^{\star}\in\R^{p}$ is a ``target'' vector we aim to recover. Interestingly, Program \eqref{prog:BP} can be solved efficiently using linear programming, e.g. \cite{MR2243152}. Furthermore, the high-dimensional models often assume that the target vector $\beta^{\star}$ belongs to the space $\Sigma_{s}$ of $s$-sparse vectors:
\[
\Sigma_{s}:=\{\beta\in\R^{p}\,,\ \lVert\beta\lVert_{0}\leq s\}\,,
\]
where $\lVert\beta\lVert_{0}$ denotes the size of the support of $\beta$. Note that this framework is the baseline of the flourishing Compressed Sensing (CS), see \cite{MR2236170,MR2241189,MR2449058,CGLP} and references therein. A breakthrough brought by CS states that if the matrix $X$ is drawn at random (e.g. $X$ has i.i.d. standard Gaussian entries) then, with overwhelming probability, one can faithfully recovers $\beta^{\star}\in\Sigma_{s}$ using \eqref{prog:BP}. More precisely, the interplay between randomness and $\ell_{1}$-minimization shows that with only $ n=\mathcal O(s\log(p/s))$, one can faithfully reconstruct any $s$-sparse vector $\beta^{\star}$ from the knowledge of $X$ and $y:=X\beta^{\star}$. Notably, this striking fact is governed by the Null-Space Property (NSP).

\begin{definition}[Null-Space Property of order $s$ and dilatation $C$]
Let $0< s< p$ be two integers and $G$ be a sub-space of $\R^p$. One says that the sub-space $G$ satisfies $\mathrm{NSP}(s,C)$, the Null-Space Property of order $s$ and dilatation $C\geq1$, if and only if:
 \eq
 \notag
  \forall\,\gamma\in G\,,\ \forall\,S\subset\{1,\ldots,p\}\ \mathrm{s.t.}\ \abs S\leq s\,,\quad C\lVert\gamma_{S}\lVert_1\leq  \lVert\gamma_{S^c}\lVert_1\,,
 \qe
where $S^c$ denotes the complement of $S$, the vector $\gamma_ S$ has entry equal to $\gamma_i$ if $i\in S$ and $0$ otherwise, and $\abs S$ is the size of the set $S$.
\end{definition}
\noindent
As a matter of fact, one can prove \cite{MR2449058} that the operator $\Delta_{X}$ is the identity on $\Sigma_{s}$ if and only if the kernel of $X$ satisfies $\mathrm{NSP}(s,C)$ for some $C>1$.

\begin{theorem}[\cite{MR2449058}]\label{thm:Cohen}
For all $\beta^{\star}\in\Sigma_{s}$ there is a unique solution to \eqref{prog:BP} and $\Delta_{X}(\beta^{\star})=\beta^{\star}$ if and only if the nullspace $\ker(X)$ of the matrix $X$ enjoys $\mathrm{NSP}(s,C)$ for some $C>1$. Moreover, if $\ker(X)$ enjoys $\mathrm{NSP}(s,C)$ for some $C>1$ then for all $\beta^{\star}\in\R^{p}$,
 \[
   \lVert\beta^{\star}-\Delta_{X}(\beta^{\star})\lVert_1\leq \dfrac{2(C+1)}{C-1}\min_{|S|\leq s}\lVert\beta^{\star}-\beta^{\star}_{S}\lVert_1\,.
  \]
\end{theorem}

\noindent
Additionally, NSP suffices to show that any solution to \eqref{prog:BP} is comparable to the $s$-best approximation of the target vector $\beta^{\star}$. Theorem \ref{thm:Cohen} demonstrates that NSP is a natural property that should be required in CS and High-dimensional statistics. This analysis can be lead a step further considering Lasso \cite{MR1379242} or Dantzig selector \cite{MR2382644}. Indeed, in the frame of noisy observations, $\ell_{1}$-minimization procedures are based on sufficient conditions like Restricted Isometry Property (RIP) \cite{MR2382644}, Restricted Eigenvalue Condition (REC) \cite{MR2533469}, Compatibility Condition (CC) \cite{van2009conditions}, Universal Distortion Property (UDP) \cite{de2012remark}, or $H_{s,1}$ condition \cite{juditsky2011accuracy}. Note that all of these properties imply that the kernel of the matrix $X$ satisfies NSP. While there exists pleasingly ingenious and simple proofs of RIP, see \cite{CGLP} for instance, a direct proof of NSP (without the use of RIP) remains a challenging issue.

\subsection{Contribution}
Given $(\rho,\delta)\in]0,1[^{2}$, set $s_{n}=\lfloor \rho n\rfloor$ and $p_{n}=\lfloor \frac n\delta\rfloor$ where $\lfloor .\rfloor$ denotes the integer part. Consider a matrix $X(n,p_{n})\in\mathds R^{n\times p_{n}}$ with i.i.d. centered Gaussian entries. In this paper, we describe a region of parameters $(\rho,\delta)$ such that $\P[\ker(X(n,p_{n}))\ \mathrm{enjoys NSP}(s_{n},C)]$ tends to one as $n$ goes to infinity. Our result provides a new and simple description of such region of parameters $(\rho,\delta)$.
\begin{theorem}
\label{thm:BorneR}
Let $C\geq1$. For all $n\geq1$, set $s_{n}=\lfloor \rho n\rfloor$ and $p_{n}=\lfloor \frac n\delta\rfloor$. Let $G(n,p_n)$ be uniformly distributed on the Grassmannian $\mathrm{Gr}_{m}(\R^{p_n})$ where $m = p_n - n$.
If $\delta\geq(1+\pi/2)^{-1}$ and:
\begin{align*}
\rho \log&\left(\sqrt{\frac{\pi}{2 e C^2}} \frac{(1-\rho)^2}{\rho^2}\right) + \log\left(C e \frac{\sqrt{\rho (1-\delta) (1 + (C^2-1)\rho)}}{ (1-\rho) (1+(2C^2-1)\rho)\sqrt{\delta}}\right) \\ &+\frac{1}{\delta} \log \left(\sqrt{\frac{2}{e \pi}} \frac{1 + (2C^2-1)\rho}{(1-\rho)\sqrt{\delta(1-\delta)(1+(C^2-1)\rho)} }\right)  \leq 0 \label{trans:1}
\end{align*}
then $\P[G(n,p_{n})\ \mathrm{enjoys NSP}(s_{n},C)]$ tends exponentially to one as $n$ goes to infinity.
\end{theorem}
\begin{figure}[!t]
\begin{center}
\includegraphics[width=0.7\textwidth]{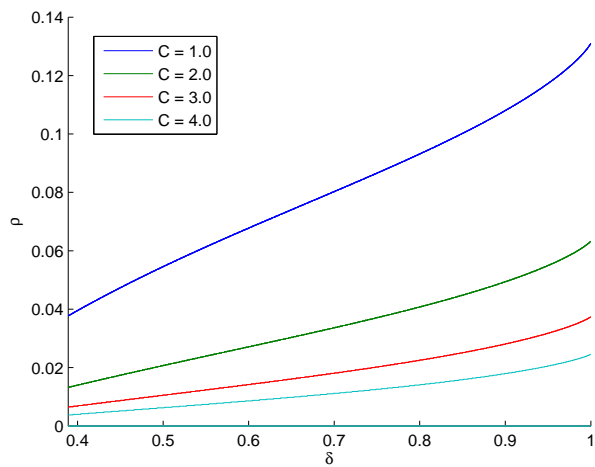}
\end{center}
\caption{The phase transition of Theorem \ref{thm:BorneR}. From top to bottom, $C=1,2,3,4$.}
\label{fig:phase}
\end{figure}
\begin{remark}
The condition 
$\{\delta\geq(1+\pi/2)^{-1}\}$
is a technical restriction. Indeed, in the proof, we need to consider an union bound on spheres of decreasing dimension. However, concentration bounds are less efficient on smaller spheres and lead to a limitation of the argument. This is explained into more details further. Interestingly, for $C=1$, the transition of Theorem \ref{thm:BorneR} compares to the phase transition of Donoho and Tanner \cite{donoho2005neighborliness}, see Figure \ref{fig:TransitionCompliquee}. However, numerically, our lower bound is less interesting when $n\ll p$ so we cannot extract the classical $n \sim c_1 s \log(c_2 p/s)$ result. To underline this fact, we only exhibit a result for $\delta \geq (1+\pi/2)^{-1}\simeq0.389$.
\end{remark}
\ni
\begin{exemple}
In the case $C=1$, we can compare our result (see Theorem \ref{thm:Main}) given by the set of $(\rho,\delta) \in ]0,1[^2$ such that $\delta \geq(1+\pi/2)^{-1}$ and such that:
\begin{align*}
\rho \log&\left[\sqrt{\frac{\pi}{2e}} \frac{(1-\rho)^2}{\rho^2}\right] + \log\left[e \frac{\sqrt{\rho (1-\delta)}}{ (1-\rho) (1+\rho)\sqrt{\delta}}\right]+\frac{1}{\delta} \log \left[\sqrt{\frac{2}{e \pi}} \frac{1 +\rho}{(1-\rho)\sqrt{\delta(1-\delta)} }\right] 
\end{align*}
is non-positive to the work of Donoho and Tanner \cite{donoho2005neighborliness} $($see Figure \ref{fig:CompPhaseDT}$)$. Observe that, up to a constant bounded by 2, we recover the initial result.
\end{exemple}
\begin{figure}[!t]
\center
\includegraphics[width=0.7\textwidth]{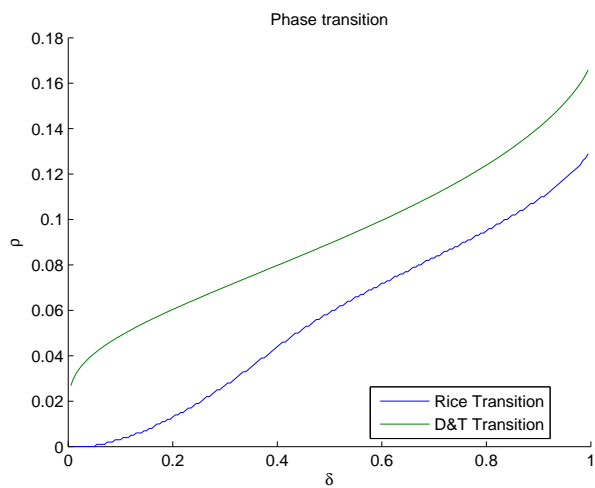}
\caption{The panel illustrates numerically the border of the region described by Theorem \ref{thm:Main} (blue line) for which NSP holds ($\Pi\simeq0$) and the strong phase transition of Donoho and Tanner (green line). Note that the region $(\rho,\delta)$ such that $\Pi\simeq0$, i.e. $\mathrm{NSP}(s,1)$ holds, is located below the curve. Simulations have been performed with $n=200.000$.}\label{fig:TransitionCompliquee}
\end{figure}

\begin{figure}[!t]
\begin{center}
\includegraphics[width=0.45\textwidth]{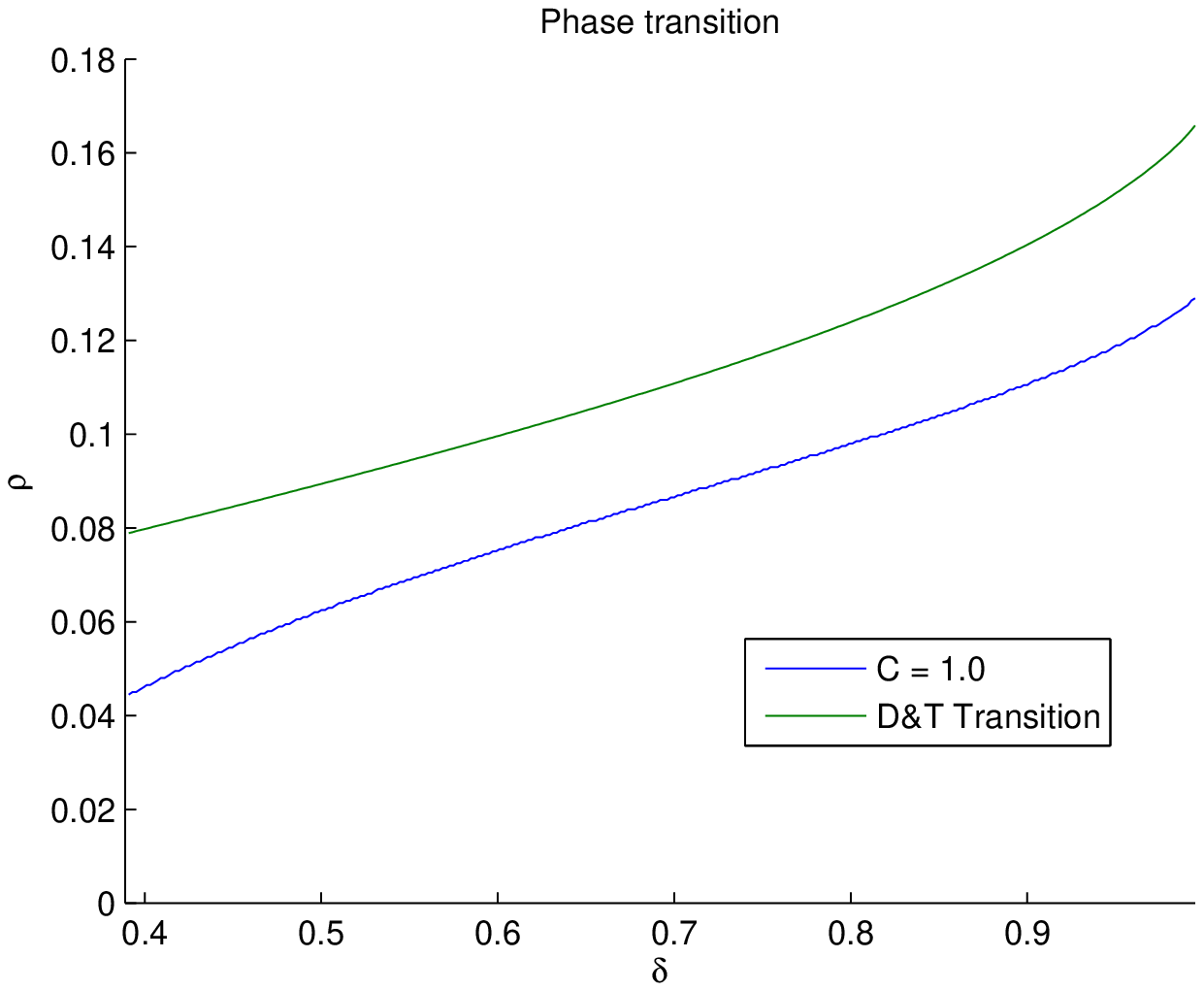}
\includegraphics[width=0.45\textwidth]{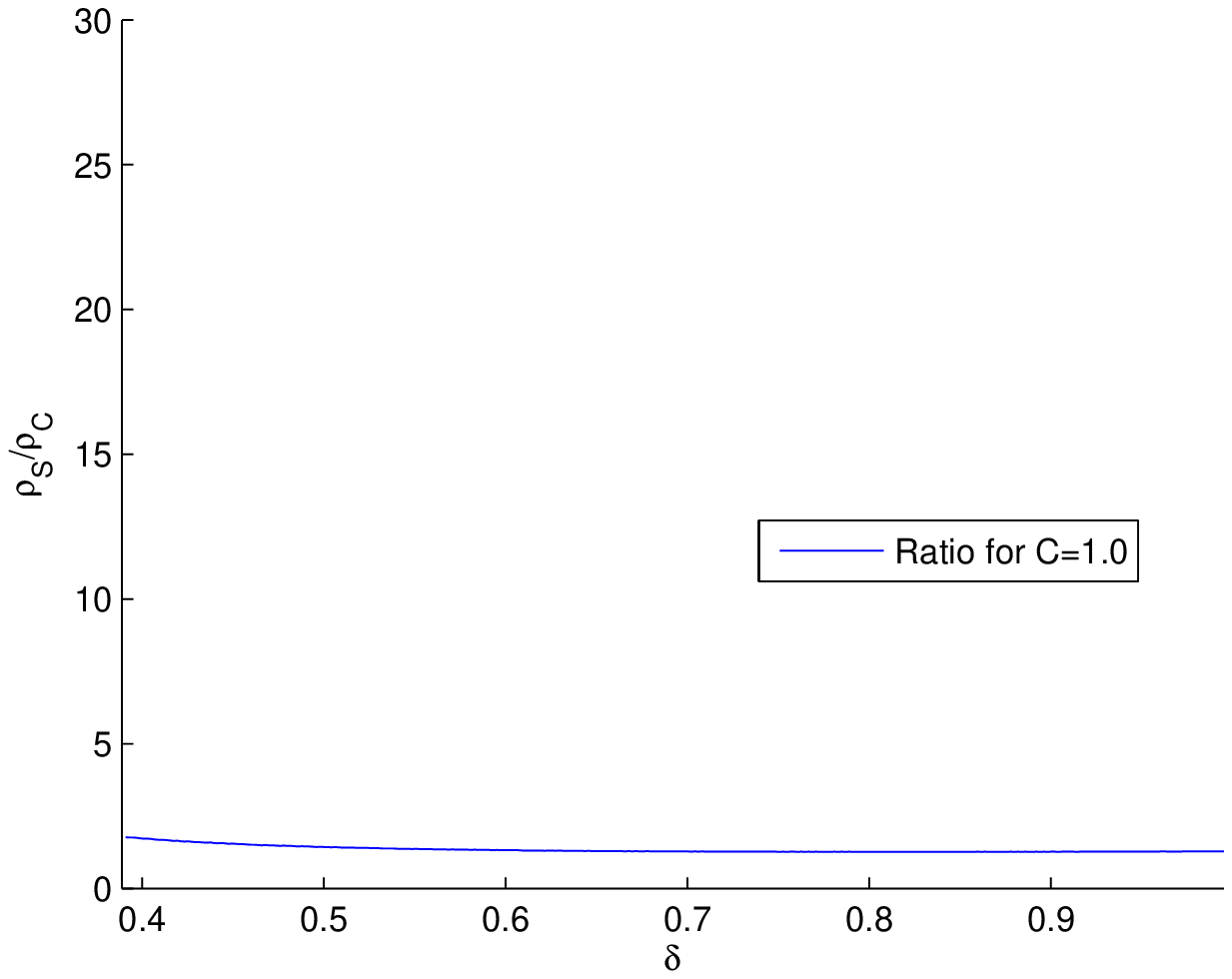}
\end{center}
\caption{On the left, comparison between the border of the region described by Theorem \ref{thm:BorneR} (blue line) and the strong phase transition of Donoho and Tanner (green line) for $\delta \geq 0.39$. On the right, ratio between the green and the blue line.}
\label{fig:CompPhaseDT}
\end{figure}

\ni
We outline that explicit expressions of lower bounds on the phase transition can be found in Section \ref{sec:Bounds}.

\subsection{Direct proofs of NSP with dilatation $C=1$}
\label{sec:DoTa}
To the best of our knowledge, all the direct proofs of NSP with dilatation $C=1$ are based either on integral convex geometry theory, Gaussian widths, the approximate kinematic formula, or empirical process theory. This section is devoted to a short review of some state-of-the-art results on direct proofs of NSP.
\subsubsection{Grassmann angles}
In a captivating series of papers \cite{donoho2005neighborliness,donoho2006high,donoho2009observed,donoho2009counting}, Donoho and Tanner have proved that the kernel of a matrix $X(n,p_{n})\in\mathds R^{n\times p_{n}}$ with i.i.d. centered Gaussian entries enjoys a phase transition, i.e. there exists a function $\rho_{S}$: $]0,1[\to]0,1[$ such that for all $(\rho,\delta)\in]0,1[^{2}$,
\[
\lim_{n \rightarrow +\infty}\P[\ker(X(n,p_{n}))\ \mathrm{enjoys NSP}(s_{n},1)] =\begin{cases}
     0 & \text{if}\ \rho>\rho_{S}(\delta), \\
     1 & \text{if}\ \rho<\rho_{S}(\delta),
\end{cases}
\]
where we recall that $s_{n}=\lfloor \rho n\rfloor$ and $p_{n}=\lfloor \frac n\delta\rfloor$. Moreover, they have characterized implicitly and computed numerically the function $\rho_{S}$ (note that the subscript $S$ stands for ``Strong'' since $\rho_{S}$ is often named the ``strong threshold''). Observe their approach is based on computation of Grassmann angles of a polytope due to Affentranger and Schneider \cite{affentranger1992random} and Vershik and Sporyshev \cite{vershik1992asymptotic}. Furthermore, note their phase transition is characterized implicitly using an equation involving inverse Mills ratio of the standard normal density. However, they have derived a nice explicit expression of the phase transition for small values of $\delta$, i.e. when $\delta\to 0$. Hence, they uncover that, in the regime $n\ll p$, NSP(s,1) holds when $n\geq C s\log(\frac ps)$ for $n$ large enough.

\subsubsection{Gaussian widths}
In recent works \cite{stojnic2009various,stojnic2013rigorous}, Stojnic has shown a simple characterization of the sign of the exponent appearing in the expression of the ``weak threshold'' given by Donoho and Tanner. Note the weak threshold governs the exact reconstruction by $\ell_{1}$-minimization of $s$-sparse vectors with prescribed support and signs, while NSP characterizes the exact reconstruction of all $s$-sparse vectors. In the paper \cite{stojnic2009various}, using "Gordon's escape through a mesh" theorem, Stojnic have derived a simpler implicit characterization of the strong threshold $\rho_{S}$. As in Donoho and Tanner's work,
observe this implicit characterization involves inverse Mill's ratio of the normal distribution and no explicit formulation of $\rho_{S}$ can be given.

Predating Stojnic's work, Rudelson and Vershynin (Theorem 4.1 in \cite{rudelson2008sparse}) were the first to use "Gordon's escape through the mesh" theorem to derive a non-asymptotic bound on sparse recovery. A similar result can found in the astonishing book of Foucart and Rauhut, see Theorem 9.29 in \cite{foucart2013mathematical}. Observe that these results hold with a probability at least $1-\alpha$ and their bounds depend on $\log(\alpha)$ so one needs one more step to derive a lower bound on the strong phase transition. We did not pursue in this direction.

\subsubsection{Approximate kinematic formula}
In the papers \cite{mccoy2012sharp,amelunxen2013living}, the authors present appealing and rigorous quantitative estimates of weak thresholds appearing in convex optimization, including the location and the width of the transition region. Recall that NSP is characterized by the strong threshold. Nevertheless, the weak threshold describes a region where NSP cannot be satisfied, i.e.
\[
\lim_{n \rightarrow +\infty}\P[G(n,p_{n})\ \mathrm{enjoys NSP}(s_{n},1)] =0\,.
\]
Based on the approximate kinematic formula, the authors have derived recent fine estimates of the weak threshold. Although their result has not been stated for the strong threshold, their work should provide, invoking a simple union bound argument, a direct proof of NSP with dilatation $C=1$.

\subsubsection{Empirical process theory}
\label{sec:Lecue}
Using empirical process theory, Lecu\'e and Mendelson \cite{guillaume2014compressed} gives a direct proof of NSP for matrices $X$ with sub-exponential rows. Although the authors do not pursue an expression of the strong threshold, their work shows that NSP with dilatation $C=1$ holds, with overwhelming probability, when:
\eq
\label{eq:NSPbound1}
n\geq c_{0}s\log(\frac{ep}{s})\,,
\qe
with $c_{0}>0$ a universal (unknown) constant. 

\subsubsection{A previous direct proof of NSP with dilatation $C\geq1$}
Using integral convex geometry theory as in Donoho and Tanner's works \cite{donoho2005neighborliness,donoho2006high,donoho2009observed,donoho2009counting}, Xu and Hassibi have investigated \cite{xu2008compressed,xu2011precise} the property $\mathrm{NSP}(s,C)$ for values $C\geq1$. Their result uses an implicit equation involving inverse Mill's ratio of the normal distribution and no explicit formulation of their thresholds can be derived. To the best of our knowledge, this is the only proof of $\mathrm{NSP}(s,C)$ for values $C>1$ predating this paper.

\subsection{Simple bounds on the phase transition}
As mentioned in Proposition 2.2.17 of \cite{CGLP}, if NSP holds then 
\eq
\label{eq:NSPbound2}
n\geq c_{1} s\log(\frac{c_{2}p}s)\,,
\qe 
with $c_{1},c_{2}>0$ are universal (unknown) constants. The result of Section \ref{sec:Lecue} shows that a similar bound is also sufficient to get NSP. What can be understood is that the true phase transition (as presented in \cite{donoho2005neighborliness,donoho2006high,donoho2009observed,donoho2009counting}) lies between the two bounds described by \eqref{eq:NSPbound1} (lower bound) and \eqref{eq:NSPbound2} (upper bound). Observe that these bounds can be equivalently expressed in terms of $\rho=s/n$ and $\delta=n/p$. Indeed, one has:
\eq
\label{eq:PhaseExpr}
\{n\geq c_{1}s\log(\frac{c_{2}p}{s})\}
\Leftrightarrow \{A_{\star}\rho\delta\log(A_{\star}\rho\delta))\geq -B_{\star}\delta\}\,,
\qe
where $A_{\star}=c_{2}^{-1}>0$ and $1/e\geq B_{\star}=c_{1}^{-1}c_{2}^{-1}>0$. Denote by $\Lambert 0$ (resp. $\Lambert{-1}$) the first (resp. the second) Lambert W function, see \cite{corless1996lambertw} for a definition. We deduce that \eqref{eq:PhaseExpr} is equivalent to:
\eq
\label{eq:PhaseExprFinale}
\rho\leq \frac{\exp(\Lambert{-1}(-B_{\star}\delta))}{A_{\star}\delta}\quad\mathrm{or}\quad\rho \geq\frac{\exp(\Lambert{0}(-B_{\star}\delta))}{A_{\star}\delta}\,.
\qe
Furthermore, the papers \cite{donoho2005neighborliness,donoho2006high,donoho2009observed,donoho2009counting} show that NSP enjoys a phase transition that can be described as a region $\rho\leq\rho_{S}(\delta)$, see Section \ref{sec:DoTa}. In particular, one can check that the region described by the right hand term of \eqref{eq:PhaseExprFinale} cannot be a region of solutions of the phase transition problem.
We deduce from \cite{CGLP,guillaume2014compressed} that $\rho_{S}$, the phase transition of Donoho and Tanner \cite{donoho2005neighborliness,donoho2006high,donoho2009observed,donoho2009counting}, can be bounded by the left hand term of \eqref{eq:PhaseExprFinale}. Hence, it holds the following result.

\begin{theorem}
\label{thm:BoundRHOS}
The strong threshold $\rho_{S}$ $($phase transition of NSP$)$ of Donoho and Tanner \cite{donoho2005neighborliness,donoho2006high,donoho2009observed,donoho2009counting} is bounded by:
\eq
\label{eq:BoundStrongTreshold}
\forall\delta\in]0,1[,\quad\frac{\exp(\Lambert{-1}(-B_{1}\delta))}{A_{1}\delta}\leq\rho_{S}(\delta)\leq\frac{\exp(\Lambert{-1}(-B_{2}\delta))}{A_{2}\delta}
\qe 
where $A_{1},A_{2}>0$ and $1/e\geq B_{1},B_{2}>0$ are universal (unknown) constants.
\end{theorem}
Although bounds \eqref{eq:NSPbound1} (lower bound) and \eqref{eq:NSPbound2} (upper bound) are known, their expressions as exponential of second Lambert W functions remain overlooked in the literature. As a matter of fact, Figure \ref{fig:Lambert} depicts a comparison between $\rho_{S}$ and:
\eq
\label{eq:LamberCurveFit}
\delta\mapsto\frac{\exp(\Lambert{-1}(-0.3394\delta))}{1.38\delta}\,,
\qe
where the strong threshold curve has been taken from \cite{donoho2005neighborliness,donoho2006high,donoho2009observed,donoho2009counting}. Roughly speaking, the curve \eqref{eq:LamberCurveFit} shows empirically that NSP holds when:
\[
n\geq 4s\log(0.7p/s)\,,
\]
for large values of $s,n,p$. Recall that it is still an open problem to find a closed form for the weak and the strong thresholds. In the regime $\delta\to0$, Donoho and Tanner \cite{donoho2005neighborliness,donoho2006high,donoho2009observed,donoho2009counting} have proved that the phase transition enjoys
\[
n\geq 2es\log(p/(\sqrt\pi s))\simeq 5.4s\log(0.6p/s)\,,
\]
in the asymptotic.
\begin{figure}[!t]
\begin{center}
\includegraphics[width=0.7\textwidth]{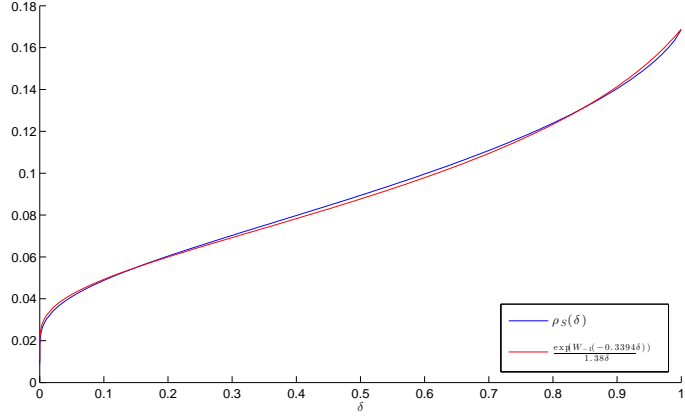}
\end{center}
\caption{The strong threshold $\rho_{S}$ and the mapping $\delta\mapsto\frac{\exp(\Lambert{-1}(-0.3394\delta))}{1.38\delta}$.}
\label{fig:Lambert}
\end{figure}

\subsection{Outline of the paper}
The main theorem (Theorem \ref{thm:Main}) is stated in the next section and Section 3 proves it. Section 4 is devoted to the proof of Theorem \ref{thm:BorneR}. All the numerical experiments can be reproduced using the codes available at \cite{MourareNSPsimu}.

\section{Rice method bound for NSP with dilatation $C\geq1$}
\label{sec:Bounds}
In this paper, we prove NSP following a newt path based on stochastic processes theory and more precisely on the Rice method \cite{azais2005distribution,azais2009level}. This latter is specially design to study the tail of the maximum of differentiable random processes or random fields. Similarly to the case of a deterministic function, it consists of studying the maximum through the zeros of the derivative. For the tail of a stationary Gaussian process defined on the real line, it is known from the work of Piterbarg \cite{piterbarg1981comparison} that it is super-exponentially sharp. 

However, the situation here is more involved than in the aforementioned papers since the considered process $X(t) $ is defined on the sphere (as in the recent work \cite{auffinger2013complexity} for example), non Gaussian and, last but not least, non differentiable. Note that the paper \cite{cucker2003expected} studies the maximum of locally linear process by a smoothing argument. A contrario to this paper, we will use a partition of the sphere and directly the Rice method. This provides a short and direct proof of $\mathrm{NSP}(s,C)$ for any value $C\geq1$. 

\subsection{An explicit sufficient condition}
Our main result reads as follows.
\begin{theorem}[Explicit lower bound]\label{thm:Main2}\label{thm:Main}
Let $0< s<n<p$ and $m=p-n$. Let $G(n,p)$ be the Kernel of $X(n,p)$, a $(n\times p)$ random matrix with i.i.d. centered Gaussian entries, then for all $C\geq1$, it holds:
\eq
\notag
\P[G(n,p)\ \mathrm{enjoys NSP}(s,C)] =1-\Pi\,,
\qe
with $\Pi$ satisfying
\begin{equation}
\label{eq:BornMoche}
\Pi\leq\sqrt \pi \Big[\sum_{k=0} ^{p-n-1}  {p \choose k}  \Big(\frac   {C^{2}s}{\tilde{p}_{C,k} }\Big) ^{\frac{ p-n-1-k}{ 2} } \frac{\G( \frac{2p-2k-n-1}{2})    }{\G(\frac{p-k}2)\G(\frac{p-n-k}2)} \psi_{p-k}(C) \mathbf{Q}(k,\tilde{p}_{C,k},m) \Big]\,,
\end{equation}
where $\G$ denotes the Gamma function, $\psi_{p-k}(C)$ is defined by Lemma \ref{lem:2}, $\mathbf{Q}(k,\tilde{p}_{C,k},m)$ is defined by Lemma \ref{lem:Zonotope} and $\tilde{p}_{C,k} := (C^2-1)s + p-k$.
\end{theorem}
\ni

\begin{figure}[!h]
\begin{center}
\includegraphics[width=0.6\textwidth]{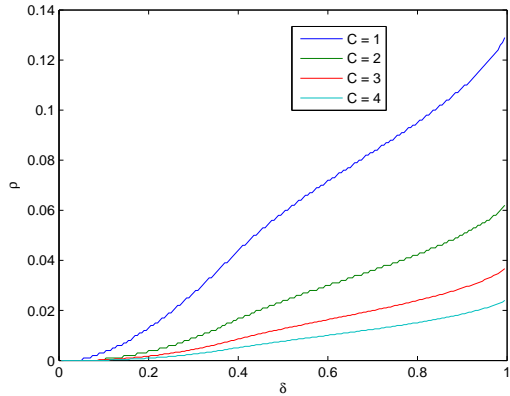}
\end{center}
\caption{Numerical computation of the lower bound $( \Pi \simeq 0)$ on the phase transition for $C=1,2,3,4$ (from bottom to top). An explicit expression can be found in Theorem \ref{thm:Main} . Simulations have been performed with $n=200.000$.}
\label{fig:Lambert2}
\end{figure}

\section{Proof of Theorem \ref{thm:Main}}
\subsection{Model and notation}
Let $0<s<n<p$, let $C>1$  and set $m=p-n$. Let $G(n,p)$ be uniformly distributed on the Grassmannian $\mathrm{Gr}_{m}(\R^p)$. Observe that it can be generated by $m$ independent standard Gaussian vectors $g_i\in\mathds R^{p}$ for $i=1,\ldots m$. Define $\{Z(t)\,;\ t \in  \mathds{S}^{m-1}\}$ the process with values in $\R^p$ given by:
 \begin{equation} 
  \notag
 Z(t) :=  \sum_{i=1}^m  t_i g_i\,.
 \end{equation}
Note this process spans $G(n,p)$ and it can be written as
\begin{equation} 
\notag
\mathrm{for}\ j=1,\ldots,p\,,\quad Z_j(t) =  \langle  t, g^j \rangle\,, 
\end{equation}
where $(g^j)_{j=1}^{p}$ are independent Gaussian random vectors with standard distribution in $\R^m$. Let $O_p$ and $O_m$ two orthogonal matrices of size, respectively, $(p\times p)$ and $(m \times m)$. Thanks to unitarily invariance of the Gaussian distribution, remark that:
\[
\forall t\in\mathds S^{m-1}\,,\quad O_p Z(t) O_m\sim Z(t)\,.
\]
Consider now the ordered statistics of the absolute values of the coordinates of $Z(t)$:
\[
|Z_{(1)} (t)| \geq  \cdots \geq |Z_{(p)} (t)|\,,
\]
where the ordering $((1),\dots,(p))$ is always uniquely defined if we adopt the convention of keeping the natural order in case of ties. Given a sparsity $s$, a degree of freedom $m$, and a degree of constraint $p$, consider the real valued process  $\{X(t)\,;\ t \in \mathds{S}^{m-1}\}$ such that:
\begin{equation}\label{f:x}
X(t) =C|Z_{(1)} (t)|  + \cdots + C|Z_{(s)} (t)|  -\big[ |Z_{(s+1)} (t)| + \cdots + |Z_{(p)} (t)| \big].
\end{equation}
NSP is equivalent to the fact that this process is always non positive.  We will prove that it happens with an overwhelming probability.

\subsection{Cutting the sphere out}
\label{sec:cut}
As we will see later, the process $X(.)$ is locally linear over some subsets and to take benefit of that, we need to consider a particular partition of the sphere.

Let $A \subseteq \{1,\ldots,p\}$, define the random subsets $\mathds S_A$ and $\dot{\mathds S}_A$ of  the unit sphere $\mathds{S}^{m-1}$ by:
\begin{align*}
\mathds S_A &= \{ t \in \mathds{S}^{m-1}\,;\ \ Z_A(t)=0 \} \,,\\
 \dot{\mathds S}_A &= \{t \in \mathds{S}^{m-1}\,;\ \ Z_A(t)=0 \ \mathrm{and}\  \forall j \notin A,\  Z_j(t) \neq 0\},
\end{align*}
where we denote $Z_{A}(t)=(Z_{j}(t))_{j\in A}$. One can check that $\mathds S_A$ is the unit sphere of the orthogonal of $V_{A} := \mathrm{Span}\{g^j\,;\ j \in A\}$. This implies that a.s. $\mathds  S_A$ is a random sphere of dimension $ m-1 -|A|$  if  $|A|\leq m-1$ and is almost surely empty if  $|A| > m-1$. It follows that the quantities  $|Z_{(1)} (t)|$, ... ,$|Z_{(n+1)} (t)|  $  are a.s. positive and that a.s. 
\[
\mathds S^{m-1}=\bigcup_{\lvert A\lvert\leq m-1} \dot{\mathds S}_A\,,
\]
giving a partition of the sphere. We define also, for later use, the random subset $\mathcal W$ by:
\[
\mathcal W:= \{ t \in \mathds S^{m-1}\,;\ |Z_{(s)} (t) | =  |Z_{(s+1)} (t) |\}\,.
\]
Observe that, conditionally to $g^{j}$, the set $\mathcal W$ is closed with empty interior. 

\subsection{Probability of failure}\label{sec:Rephrase}
We consider the probability: 
\begin{equation}\label{e:pi}
\Pi = \P \big\{ \mathcal{ M} > 0\big\}\leq
 \sum_{\lvert A\lvert\leq m-1}    \P \big\{  \mathcal{ M}_{\dot{\mathds S}_A}  > 0\big\}\,,
\end{equation}
where $ \mathcal{ M}  $ and $\mathcal{ M}_{\dot{\mathds S}_A} $ are respectively the number of positive local maximum of $X(.)$ along  $\mathds{S}^{m-1}$ and 
$\dot{\mathds S}_A$. The baseline of our proof is to upper-bound each right hand side probabilities, using the expected number of positive local  maximum above zero and  Markov inequality. The first element is Lemma \ref{lem:2} proving that:
\[
  \forall t\in \dot{\mathds S}_A \,,\quad \P\{X(t)  > 0\} \leq \psi_{p-k}(C) \,,
\]
where $k := |A|$ and:
$$\psi_{p-k} (C) = {p-k \choose s} \left(\frac{C^2 4 s}{\pi}\right)^\frac{p-k-s}{2} \frac{\G((p-k)/2)}{\G(s/2) \G(p-k-s+1)},$$
where $\G$ denotes the Gamma function. The second element is that $X(t)$ admits a density $p_{X(t)}$. To check that, note that $|Z_{(1)}(t)|,\ldots |Z_{(p)}(t)|$  are the order statistics  of the absolute values  of i.i.d. Gaussian variables  and thus they have a known joint density  on the simplex $|Z_{(1)}|\geq \ldots \geq |Z_{(p)}|$. Formula \eqref{f:x} implies the existence of a density for $X(t)$. Moreover, this density does not depend on $t$ due to invariance of Gaussian distribution.

\subsection{Initialization: local maxima on $\dot{\mathds S}_{\emptyset}$ }
By considering the symmetry properties of the sphere $\dot{\mathds S}_{\emptyset}$, we  have:
\[
\P \big\{  \mathcal{ M}_{\dot{\mathds S}_{\emptyset}}  > 0\big\} \leq  \frac12\E( \mathcal{ M}_{\dot{\mathds S}_\emptyset} )\,.
\]
 In this part, our aim will be to give bound to the expectation using a Kac-Rice formula.  One can check that if $t$  belongs to $\dot{\mathds S}_\emptyset$ and  does not belong to $\mathcal W$, $X(.)$ is locally the sum of the absolute values of some $s$  coordinates multiplied by $C$ minus the sum of the absolute values of the other coordinates.
It can be written as:
\[
  X(u) = C \varepsilon_1 Z_{(1)} (u) + \cdots + C \varepsilon_s Z_{(s )} (u)+\varepsilon_{s+1} Z_{(s+1)} (u)+\cdots + \varepsilon_p Z_{(p)} (u),
\]
  where $\varepsilon_1,...,\varepsilon_p$ are random variables taking values $\pm1$. 
  
 \begin{lemma} 
\label{lem:IndependentTangentSpace}
Let $t\in\mathds S^{m-1}$ then, almost surely, it holds $t \in \dot{\mathds S}_{\emptyset}$  and $ t\notin \mathcal W$. Furthermore, the spherical gradient $X'(t)$  and the spherical  Hessian  $X''(t)$ of $X(.)$  along $\mathds S^{m-1}$ at $t$  exist and: 
 \begin{itemize}
 \item $  X''(t)  = -X(t) I_{m-1}$.
 \item$\big( X(t), X''(t)\big) $ and $X'(t)$ are independent.
 \item $X'(t)$   has a Gaussian  centered isotropic  distribution onto $t^{\bot}$ with variance $ \tilde{p}_{C,0} = (s C^2 + (p-s))$ .
 \end{itemize}
 \end{lemma}
\begin{proof}

 The fact that, with probability 1, $t \in \dot{\mathds S}_{\emptyset}$  and $ t\notin \mathcal W$  implies that the process $X(.)$  is locally  given by 
\[
X(u) = C \varepsilon_1 Z_{(1)} (u) + \cdots + C \varepsilon_s Z_{(s )} (u)+\varepsilon_{s+1} Z_{(s+1)} (u)+\cdots + \varepsilon_p Z_{(p)} (u),
\]
  where the  signs  $(\varepsilon_1,...,\varepsilon_p)$ and the ordering $ (1), \ldots,(p)$  are those of $t$.  The process $X(.)$ 
 is  locally linear and thus  differentiable around $t$ 
 and  its gradient in $\R^m$ at $t$, denoted  $\dot{X}(t)$, is given by 
\[
C \varepsilon_1 g^{(1)} + \cdots + C \varepsilon_s g^{(s )}+\varepsilon_{s+1} g^{(s+1)}+\cdots + \varepsilon_p g^{(p)}.
\]
 Moreover, note that its Hessian on $\R^m$ vanishes. 
 
Let us consider now the spherical gradient $X'(t)$ and the spherical Hessian $X''(t)$.
It is well known that 
$
X'(t) = P_{t^{ \bot}} \dot{X}(t),
$ where  $P_{t^{ \bot}}$ is  the orthogonal projection onto the orthogonal of $t$.
As for the spherical	Hessian, it is defined on the tangent space $ t^{ \bot}$ and is equal to  the projection of the  Hessian  in $\R^m$, which vanishes, minus  the product  of the normal derivative by   the identity matrix.  This is detailed in Lemma \ref{spherical}. 
 In the case of  the unit sphere,  the vector normal  to the sphere at $t$ is $t$ itself  and 
 $$
  X''(t)  =  - \langle  \dot{X} (t),t \rangle I_{m-1} = -  X(t) I_{m-1}.
 $$
In the case of $X'(t)$, remark that $Z(t)$ and thus $ X(t)$, $(\varepsilon_1,\ldots,\varepsilon_p,(1),\ldots,(p))$  are functions  of  $(P_{t}(g^{1}),\ldots,P_{t}(g^{p}))=(Z_{1}(t)t,\ldots,Z_{p}(t)t)$ (with obvious notation). They are therefore independent of $X'(t)$ which is a function  of  $(P_{t ^\bot}(g^{1}),\ldots,P_{t^\bot}(g^{p}))$. Conditionally  to $ (\varepsilon_1,\ldots,\varepsilon_p,(1),\ldots,(p))$, $X'(t) $ can be written as 
  \[
 X'(t)= (C \varepsilon_1P_{t^{\bot}} g^{(1)} + \cdots + C \varepsilon_s P_{t^{\bot}}g^{(s )}+\varepsilon_{s+1}P_{t^{\bot}} g^{(s+1)}+\cdots + \varepsilon_p P_{t^{\bot}}g^{(p)})\,,
\]
which implies that the conditional distribution of $X'(t)$ is  Gaussian  with variance-covariance matrix  $(s C^2 + (p-s))\mathrm{Id}_{t^{\bot}}$, where $\mathrm{Id}_{t^{\bot}}$ is the identity operator on $t^{\bot}$. Since $X'(t)$ is independent of $ (\varepsilon_1,\ldots,\varepsilon_p,(1),\ldots,(p))$ this conditional distribution is in fact  equal to the unconditional distribution. 
  \end{proof}
\ni
The next step is to prove that a.s. there is no local maximum on $\mathcal W$. The case where there are tied  among the $|Z_i(t)|$  has to be considered (though it happens with probability $0$ for a fixed $t$). Note that the order statistics and the ordering remain uniquely defined because of our convention. 

\ni
 Suppose that $t\in \mathcal W$. Since  all the possible ordering $((1) ,\ldots,(p))$ and signs $(\varepsilon_1,\ldots,\varepsilon_p)$ play the same role by unitarily invariance of the distribution of $Z(t)$  for all $t$, we make the proof in the particular case where  $((1) ,\ldots,(p))$ is the identity  and all the signs  $(\varepsilon_1,\ldots,\varepsilon_p)$ are positive:
 \[ Z_1(t) \geq ...\geq  Z_{s-h-1}(t) > Z_{s-h}(t)= \ldots 
 = Z_{s+k} (t)>Z_{s+k+1}(t) \geq ..\geq Z_p(t) >0.
 \]
  Then,  for $w$ in some neighborhood  $N$ of $t$ (not included in $ \mathcal W$), we have:
  \begin{multline} 
  \notag
     X(w)  =  C Z_1 (w)+\cdots + CZ_{s-h-1} (w) + ( 1+C)  \mathrm{Max}_h  \big(Z_{s-h}(w) + \cdots +  Z_{s+k}(w) \big)  \\   - (  Z_{s-h}(w) + \cdots +  Z_{p}(w) ),
  \end{multline}
where $\mathrm{Max}_h$ is the sum of the $h$ largest element of its $(h+k +1)$ arguments. As being the maximum   of $ h \choose {(s+k)+1} $  linear forms the function $\mathrm{Max}_h$  is convex.
   
\ni
Let us consider in detail the vectors $g^{s-h}, \ldots, g^{s+k}$. With probability $1$, they are pairwise different.  The point $t$ is chosen  such that their projection  on $t $ coincide.  As a consequence the derivatives of the linear forms $ Z_{\ell}(w) = \langle g^{\ell }, w\rangle ,\ \ell = (s-h) \ldots (s+k)$ on the tangent space $t^\bot$ are pairwise different. This  implies that the function $\mathrm{Max}_h$ has some direction in which it is strictly  convex  and as a consequence $t$ cannot be a local maximum.

\ni
Suppose  that $t\notin W $, and suppose that we limit our attention to  points $t$ such that 
 $X(t)>0$,  then Lemma \ref{lem:IndependentTangentSpace}  implies that $X''(t)$  cannot be singular. 
 
This last condition implies that we can apply Theorem 5.1.1 of \cite{adler1981geometry}. This lemma is a Kac type formula that shows that  the   zeros of the derivative $X'(t)$ are isolated  an thus in finite number. In addition recalling that  $\mathcal{M}_{\dot{\mathds S}_\emptyset}$   is   the number of positive local maximum of $X(.)$ and belonging  to $\dot{\mathds S}_\emptyset$, this number satisfies 
\eq
\notag
  \mathcal{M}(\dot{\mathds S}_\emptyset) = \lim_{\delta \to 0 } \frac {1}{V(\delta)} 
  \int_{\mathds S^{m-1}} \E(|\det X''(t) |\UN_{|X'(t)-0|<\delta}  \UN_{ t \in\dot{\mathds S}_\emptyset} \UN_{X(t) >0}) \sigma(\d t),
\qe
where $\sigma$ is the surfacic measure  on $\mathds S^{m-1}$ and $V(\delta)$ is the volume of the ball $B(\delta)$ with radius $\delta$. Passing to the limit using the Fatou lemma  gives:
   \begin{align}
   \notag
  \E (\mathcal{M}(\dot{\mathds S}_\emptyset)) & \leq \liminf_{\delta \to 0} \int_0^\infty \d x \int_{\dot{\mathds S}_\emptyset}\d t\, p_{X(t)} (x)    \\ \notag
  &\frac {1}{V(\delta)} \int _{B(\delta)}  \d x' p_{X'(t)}(x')  \E\big( |\det(X''(t)) |\ \Big|   X(t)=x,X'(t)=x') \\\notag
&\leq (2\pi \tilde{p}_{C,0}\big) ^{\frac{1-m}2} 2 \frac{ \pi^{\frac m2}}{\G(\frac m2)} \int_0^\infty x^{m-1}    p_{X(t)} (x)dx  \,,
\end{align}
where  $ p_{X(t)}(x)$ denotes  the density of $X(t) $ at $x$ and $\G$ denotes the Gamma function. Note that we have used:
 \begin{itemize}
 \item  the fact that every point $t$ is equivalent so we can replace the integral on the unit  sphere by  the volume of the unit sphere $2{ \pi^{\frac m2}}/{\G(\frac m2)}$ and the value at a given point, 
 \item $\E\big( |\det(X''(t)) |\ \lvert  X(t)=x,X'(t)=x')= x^{m-1} $,
 \item  the Gaussian density $p_{X'(t)}(x')$ is bounded by $ (2\pi \tilde{p}_{C,0}\big) ^{\frac{1-m}2} $.

 \end{itemize} \medskip
So it remains to bound $\E [(X(t)^+)^{m-1}]$. For that purpose  we write $X(t)$ as the independent product $\|Z(t)\|_2 Y(t)$, where  the process $Y(t)$ is constructed exactly as the process $X(t)$ but starting now from a uniform distribution $U$  on the unit sphere $\mathds{S} ^{p-1}$ instead of the standard Gaussian distribution of $Z(t)$. Using standard results on the moments  of the $\chi^2$ distribution we have:
\[
 \E ((X(t)^+)^{m-1}) = 2^{\frac {m-1}2} \frac{\G( \frac{m-1+p}{2})}{\G(\frac p2)}   \E ((Y(t)^+)^{m-1}) \,.
\]
We use now the fact that   $Y(t) \leq C \sqrt{s}$ to get that:
\[
\E ((Y(t)^+)^{m-1})\leq (C \sqrt{s})^{m-1}\P \{Y(t) >0\}\,.
\]
Moreover, Lemma \ref{lem:2} shows that, with probability  greater than  $1 -\psi_p (C) $,  a standard Gaussian vector $g$ in $\R^p$ enjoys:
\[
 C \|g_S\|_1  \leq \|g_{S^c}\|_1\,.
\]
 This implies that:
 \eq
 \label{eq:BoundPsi}
 \P \{Y(t) >0\} \leq  \psi_p(C)\,,
 \qe
 and consequently the probability of having a local maximum above 0  on $\dot{\mathds S}_\emptyset$ is bounded by:
\eq 
\label{eq:BorneH} 
\sigma (\dot{\mathds S}_\emptyset)\,(2\pi \tilde{p}_0\big) ^{\frac{1-m}2}  2^{\frac{m-1}2} \frac{\G( \frac{m-1+p}{2})    }{\G(\frac p2)}  (C \sqrt{s})^ {m-1}    \psi_p(C)  \leq 2 \sqrt \pi \Big(\frac{C^2 s}{\tilde{p}_{C,0}}\Big) ^{\frac{ m-1}{ 2} } \frac{\G( \frac{m-1+p}{2})    }{\G(\frac p2)\G(\frac m2)}    \psi_p(C) 
\qe
Denote the right hand side of this last inequality by $h_C(s,m,p)$. 

\subsection{Maximum on smaller spheres} 
Let us now consider the case of a maximum on $\dot{\mathds S}_A$, $A \neq \emptyset $. A point $t \in \dot{ \mathds S}_A \backslash \mathcal{ W}$  is a local maximum on $\mathds{S}^{m-1}$ if it satisfies the following conditions:
 \begin{itemize}
 \item it is a local maximum along  $\mathds S_A$, \label{condi:1}
 \item its super-gradient along the orthogonal space $V_A$ contains zero, 
 \end{itemize}  
where the super-gradient is defined as the opposite of the sub-gradient. One can easily check that the two conditions are independent. Indeed, recall that $k= |A|$ and $V_A =  \mathrm{Span}\{g^i;\ i\in A\}$ (see Section \ref{sec:cut}) and consider the process $X(.)$ conditionally to $V_A$. In that case, $\mathds S_A$ becomes a deterministic sphere of dimension $m-k-1$. Moreover, note that the behavior  of $X(.)$ on $ \mathds S_A$ depends only on the $g^j, j\ \notin A$ and that for such $j$, 
\[
     Z_j(t) =   \langle g^{j},t \rangle   =    \langle \Pi_{V_A^{\bot}}g^{j},t \rangle,
\]  
so, conditionaly to $V_A$, the distribution of $X(t)$ corresponds to the case $\mathds S_\emptyset$ in the space of dimension $m-k$ instead of $m$ and with $p-k$ vectors. In conclusion, the first condition leads to the same computations as the case $\mathds{S}_\emptyset$ and is bounded by
 
 \begin{align} \notag
     h_C(s,m-k,p-k) = &2\sqrt \pi \Big(\frac   {C^{2}s}{\tilde{p}_{C,k}}\Big) ^{\frac{ p-n-1-k}{ 2} } \frac{\G( \frac{2p-2k-n-1}{2})    }{\G(\frac{p-k}2)\G(\frac{p-n-k}2)} \psi_{p-k}(C)\,.\notag
 \end{align} 
Let us look to the second one which depends only on the $g^j, j \in A$. Thus we have to compute  the  probability  of the super-gradient to contain zero. Indeed, locally around $t$,  the behavior of  $X(w) $  along  $V_A$ is the sum of  some linear forms (for $ \ j \notin A$) and of  absolute value of linear forms (for $ \ j \in A$)  thus it is locally concave and  we can define its super-gradient. More precisely, for $w$ in a  neighborhood of $ t\in \dot{\mathds S}_{A}\setminus \mathcal W$,    
\[
X(w)  = X_A(w) + X_{A^c}(w)\,,
\]
where,  because $k \leq p-s$:
\[
    X_A(w) =  - \sum_{i \in A} |Z_i (w)| \,.
\]
Around $t$, $X_{A^c}(w)  $ is differentiable and, with a possible harmless change of sign (see Lemma \ref{lem:IndependentTangentSpace}), its gradient is given by:
\[
\sum_{i \in A^c}  C_i g^i\,,
\]
where the coefficient $  C_i $  takes the value $C$ for $s$ of them and $-1$ for the others.  This gradient is distributed as an isotropic normal variable $\xi \in V_A $ with variance:
\[
\tilde{p}_{C,k}=(C^2-1) s +p-k\,.
\]
By this we mean that the distribution of  $\xi$, in a convenient basis, is $\mathcal N(0,  \tilde{p} _{C,k} I_k)$.
\ni
Let us now consider the case $i \in A$. Observe that the super-gradient along $V_A$ of the concave function $-|Z_i(t)|$ at point $t$ is the segment $[-g^i,g^i]$ and thus the super-gradient of $X_A(t)$ is  the zonotope:
  \begin{equation}\label{f:zozo}
     Zo = \sum_{i \in A} [-g^i,g^i] ,  
      \end{equation}
   where the sum denotes the Minkowsky addition. Recall that the distribution of $X(t)$ does not depend on $t$.
   
   \medskip
   In conclusion, the probability of the super-gradient to contain zero is equal to $ \mathbf{P}(k,\tilde{p}_{C,k},m )$  
the probability  of the following event:
     \begin{itemize} 
      \item draw $k$ standard Gaussian  variables   $ g^1,\ldots,g^k$  in $\R^m$  and consider  the zonotope  $Zo$ given by formula  \eqref{f:zozo},
      \item  draw in the space  $V_A$ generated by $g^1,\ldots, g^k$ 
         an  independent isotropic normal variable $\xi$  of variance $ \tilde{p}_{C,k}$,
      \item  define  $\mathbf{P}(k,\tilde{p}_{C,k},m)$ as the probability of $\xi$ to be  in $Zo$.
\end{itemize}
\begin{lemma} \label{lem:Zonotope}
 Define the  orthonormal basis $e^1,\ldots,e^k$ obtained by Gram-Schmidt  orthogonalization of the vectors $g^1,\ldots g^k$.  Then:
 \begin{enumerate}
 \item[$\mathrm{(a)}$]
$\mathbf{P}(k,\tilde{p}_{C,k},m)$ is less than the probability   $ \mathbf{Q}(k,\tilde{p}_{C,k},m) $ of $\xi$ to be  in the  hyper-rectangle:
  \begin{equation}
  \notag
     R = \sum_{i \in A} [- \langle e^i, g^i\rangle e^i,\langle e^i, g^i\rangle e^i] ,  
      \end{equation}
\item[$\mathrm{(b)}$]
this last probability satisfies:
      $$
      \big(\mathbf{Q}(k,\tilde{p}_{C,k},m)  \big)^2\leq 
  \left(\frac{2}{\pi \tilde{p}_{C,k}}\right)^{H_k+k-m}\frac{ H_k!}{(m-k)!},
   $$
   with $ H_k =\lfloor (\frac\pi2  \tilde{p}_{C,k})\wedge m\rfloor$, where $\lfloor.\rfloor$ is the integer part.
 \end{enumerate}
 \end{lemma}
      
\begin{figure}
\center
\begin{tikzpicture}[scale=4]
\draw[black,dashed,->] (0.2,-0.45)--(0.2,0.45);
\draw[black,dashed,->] (0,0) -- (2,0) ;
\draw[RoyalBlue,very thick] (0.75,0.4) -- (1.75,0.4) -- (1.25,-0.4) -- (0.25,-0.4) --(0.75,0.4);
\draw[Bittersweet,very thick] (0.5,0.4) -- (1.5,0.4)-- (1.5,-0.4)-- (0.5,-0.4) -- (0.5,0.4) -- (0.5,0.4) ;
\draw[OliveGreen,very thick,|-|] (0.5,0)--(1.5,0);
\node[text=OliveGreen,font=\Large] at (1,-0.1) {$\tilde{Zo}_{k-1}$};
\node[text=black,,font=\large] at (1.8,-0.1) {$ z_1,\ldots,z_{k-1}$};
\node[text=black,,font=\large] at (0.1,0.4) {$ z_k$};
\end{tikzpicture}
\label{fig:Zozolezono}
\caption{The standard Gaussian measure of the zonotope (in blue) is smaller than that of the rectangle (in red) with basis $\tilde{Zo}_{k-1}$ (in green).}
\end{figure}
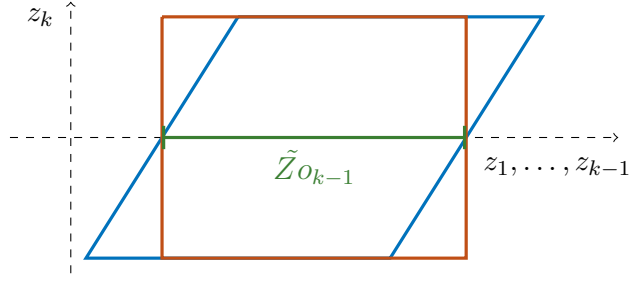
       
       \begin{proof}
       
      (a)  We prove the result  conditionally to the $g^i$'s and by induction on $k$. When $k=1$ the result is trivial since the zonotope and the rectangle  are simply the same segment. 
      
\ni     
Let $\varphi_h $ be  the standard Gaussian distribution on $\R^h$,  $\mathbf{P}(k,\tilde{p}_{C,k},m)$ is equal to:
\[
\varphi_k \big(  (\tilde{p}_{C,k})^{-1/2}. Zo\big) =: \varphi_k\big(\widetilde{Zo}\big).
\]
Via Gram-Schmidt ortogonalisation at step $k$, we can compute this probability using the Fubini theorem:
\[
 \mathbf{P}(k,\tilde{p}_{C,k},m)=  \int _{-\frac{\langle e^k,g^k\rangle}{\sqrt{\tilde{p}_{C,k}}}}
  ^{\frac{\langle e^k,g^k\rangle}{\sqrt{\tilde{p}_{C,k}} }  }  \varphi_{k-1}\big( \widetilde{Zo}_{k-1}  + v z \big)  \varphi(z)dz,
\]
where $\varphi$ is the standard Gaussian density on $\mathds R$, $\widetilde{Zo}_{k-1}$ is the zonotope  generated by $g^1,\ldots,g^{k-1}$  and normalized by $ (\tilde{p}_{C,k})^{-1/2}$ and  $v$ is some vector in $\mathds R^m$. By use of the Anderson inequality \cite{anderson1955integral}, the non-centered zonotope $\big( \widetilde{Zo}_{k-1}  + vz \big)$ has a smaller  standard Gaussian measure  than the centered one so
\begin{align*}
\mathbf{P}(k,\tilde{p}_{C,k},m)\ &\leq  
\int _{-\frac{\langle e^k,g^k\rangle}{\sqrt{\tilde{p}_{C,k}}}}
  ^{\frac{\langle e^k,g^k\rangle}{\sqrt{\tilde{p}_{C,k}} }  } 
    \varphi_{k-1}\big( \widetilde{Zo}_{k-1}   \big) \varphi(z) dz \\
&\leq  \int_{R}  \varphi(z_ 1) \ldots \varphi(z_k)  dz_1\ldots dz_k =: \mathbf{Q}(k,\tilde{p}_{C,k},m).
\end{align*}
The last  inequality  is due to the induction hypothesis. It achieves the proof. 

 (b) We  use  the relation above  and  deconditioning  on the $g^i$. Note the dimension of the  edges of the rectangle $R$ are independent with distribution: 
 \[
 2\chi(m), 2\chi(m-1),\ldots, 2\chi( m-k+1)\,,
 \] 
 where  the law $\chi(d)$ is defined as the square root of a $\chi^2(d)$. As a consequence, using the independence of  the components of $\xi$ in the basis $e^1,\ldots,e^k$ and the fact that a Student density $T$ is uniformly bounded  by $(2 \pi ) ^{-1/2}$, we get that:
\begin{align*}
 \mathbf{Q}(k,\tilde{p}_{C,k},m) =  \mathds{P}(\xi \in R) & = \prod_ {\ell=0}^{k-1} \P\Big[ | T (m-\ell)|  \leq \sqrt{\frac{m-\ell  }{\tilde{ p }_{C,k}}} \Big] \\
 &= \prod_ {\ell=m-k+1}^{m} \P\Big[ | T (\ell)|  \leq \sqrt{\frac{\ell  }{\tilde{ p }_{C,k}}} \Big]\,.
\end{align*}
  Suppose  that $ \pi \tilde{ p }_{C,k}  \geq 2m$, then a convenient bound is obtained by using the fact that a Student density  is uniformly bounded  by $(2 \pi ) ^{-1/2}$: 
\[
 \big( \mathbf{Q}(k,\tilde{ p }_{C,k},m) \big)^2\leq 
  \left( \frac{2}{\pi\tilde{ p }_{C,k}} \right)^{k}\frac{ m!}{(m-k)!}.
\]
In the other  case, set $ H_k = \lfloor(\pi  \tilde{ p }_{C,k})/2\rfloor$, where $\lfloor.\rfloor$ is the integer part. Observe that $H_{k}> m-k+1$ for $k\geq1$ to remove  factors that are greater than 1 in the computation and obtain
\[
 \big( \mathbf{Q}(k,\tilde{ p }_{C,k},m) \big)^2\leq 
  \left( \frac{2}{\pi\tilde{ p }_{C,k}} \right)^{H_k+k-m}
  \frac{ H_k!}{(m-k)!}\,,
\]
which conclude the proof. 
\end{proof} 
\ni
Eventually, summing up  over  the ${p \choose k}$ sets of size $k$, we get Theorem \ref{thm:Main2}. 

\section{Influence of smaller spheres}
\subsection{General bound on the sum}
In this part, we simplify the general bound of Theorem \ref{thm:Main} to derive a simpler one, exponentially decresing in $n$, as presented in Theorem \ref{thm:BorneR}. Considering \eqref{eq:BornMoche}, we have:
\eq
\Pi \leq 
\sqrt{\pi} \sum_{k=1}^{m} B_k(s,n,p)
\qe
where:
\begin{align*} B_k(s,n,p) = {p \choose n+k} {n+k \choose s} & \left(\frac{C^2 s}{\tilde{p}_{m-k}}\right)^\frac{k+1}{2} \frac{\G(\frac{n}{2} + k-\frac12)}{\G(\frac{n+k}{2}) \G(\frac{k}{2})} \\ & \times \left(\frac{C^2 4 s}{\pi}\right)^{\frac{n+k-s}{2}}  \frac{\G(\frac{n+k}{2})\sqrt{H_{m-k}!}}{\G(\frac{s}{2}) (n+k-s)!\sqrt{k!}} \left( \frac{2}{\pi \tilde{p}_{C,m-k}}\right)^\frac{H_{m-k} - k}2
\end{align*}
In order to derive a lower bound (the aforementioned bound goes exponentially fast towards zero), we limit our attention to the case described by
\begin{itemize}
\item (H1) $\rho \leq 1/2$,
\item (H2) $\frac{1}{\delta} \leq 1 + \pi/2(1 + \rho (C^2 - 1))$.
\end{itemize}
Observe that (H1) is not a restriction since we know that NSP does not hold for $\rho\geq0.2$.
Under (H2), note that $\forall k, ~H_k = m$, and hence
\begin{align*}
\Pi &\leq \mathbf R(s,n,p) \frac{p!}{s! \sqrt{m!}\G(s/2)} \left(\frac{4 C^2 s}{\pi}\right)^\frac{n-s}{2}\sum_{i=1}^{m}{m \choose k} \frac{(2 C^2 s)^k\G(\frac{n}{2} + k)}{(n-s+k)!^2} \left( \frac{2}{\pi \tilde{p}_{C,m-k}}\right)^\frac{m}2 \\
&\leq \mathbf R(s,n,p) \frac{p!}{s! (n-s)! \sqrt{m!}\G(s/2)} \left(\frac{4 C^2 s}{\pi}\right)^\frac{n-s}{2}\\&
\quad\quad\quad\quad\quad\times\sum_{i=1}^{m}{m \choose k} \frac{(2 C^2 s)^k\G(\frac{n}{2} + k)}{(n-s)^k(n-s+k)!} \left( \frac{2}{\pi \tilde{p}_{C,m-k}}\right)^\frac{m}2,
\end{align*}
where $\mathbf R(s,n,p)$ is a polynomial term in $(s,n,p)$. 
Consider now the quantity 
\[\alpha(k) := \frac{\G(\frac{n}{2}+k)}{(n-s+k)!}\]
which is a decreasing function of $k$ under assumption $(H1)$ and the fact that $\tilde{p}_{C,k}$ is an increasing function of $k$, then we obtain
\[\Pi \leq \mathbf R(s,n,p) \frac{p! \G(n/2)}{s! (n-s)!^2 \sqrt{m!}\G(s/2)} \left(\frac{4 C^2 s}{\pi}\right)^\frac{n-s}{2}\sum_{i=1}^{m}{m \choose k} \frac{(2 C^2 s)^k}{(n-s)^k} \left( \frac{2}{\pi \tilde{p}_{C,m-1}}\right)^\frac{m}2 .\]
At last, using Stirling Formula (see Lemma \ref{lem:AsymptoticGamma}), it yields
\[ \Pi \leq \mathbf R(s,n,p) \frac{p! \G(n/2)}{s! (n-s)!^2 \G(s/2)} \left(\frac{4 C^2 s}{\pi}\right)^\frac{n-s}{2} \left(\frac{2 e}{\pi (n + (C^2-1)s) m}\right)^\frac{m}{2} \left(1 + \frac{2 C^2 s}{n-s}\right)^m.
\]
Gathering the piece, one has: 
\begin{align*} \Pi \leq \mathbf R(s,n,p) \left(\sqrt{\frac{\pi}{2 e C^2}} \frac{(n-s)^2}{s^2}\right)^s &\left(C e \frac{\sqrt{n s m (n + (C^2-1)s)}}{(n-s)(n+(2C^2-1)s)}\right)^n \\ & \times \left(\sqrt{\frac{2}{e \pi}} \frac{p (n+(2C^2-1)s)}{(n-s)\sqrt{m(n+(C^2-1)s)}}\right)^p \,,
\end{align*}
which gives the result of Theorem \ref{thm:BorneR}.
\begin{remark}
The upper bound on $\alpha(k)$ and the lower bound on $\tilde{p}_{C,k}$ may seem weak but they do not change the result on the phase transition because first terms of the sum give the right order on $\rho$ and $\delta$. To ensure that, see figure \ref{fig:CompPhase}, which compare the numerical bound with the one of Theorem \ref{thm:BorneR}. 
\end{remark}
\begin{figure}[!t]
\begin{center}
\includegraphics[width=0.7\textwidth]{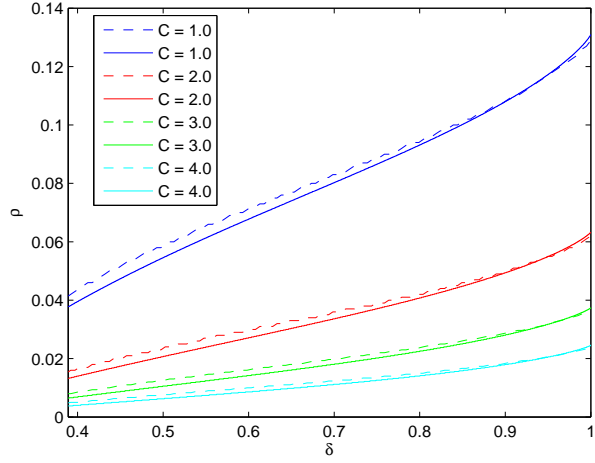}
\end{center}
\caption{Comparison between the phase transition of Theorem \ref{thm:BorneR} and the numerical approximation given by Theorem \ref{thm:Main} ( $\Pi \simeq 0$ with $n=200.000$, dashed line). From top to bottom, $C=1,2,3,4$.}
\label{fig:CompPhase}
\end{figure}

 \appendix
\section{Stirling's formula}
\begin{lemma}
\label{lem:AsymptoticGamma}
Let $z>0$ then there exists $\theta\in(0,1)$ such that:
\[
\G(z+1)=(2\pi z)^{\frac12}\left(\frac ze\right)^{z}\exp(\frac\theta{12z})\,.
\]
In particular, if $z>1/12$,
\[ \left(\frac{z}{e}\right)^z \leq \G(z+1) \leq \sqrt{2 \pi z} \left(\frac{z}{e}\right)^z\]

\end{lemma}
\begin{proof}
See \cite{abramowitz1965stegun} Eq. 6.1.38.
\end{proof}
\section{Concentration $\psi_l(C)$}
 \begin{lemma} \label{lem:2}
Let $ C\geq 1$, then, except with a probability smaller than:
\eq
\psi_l ( C) :=  {l \choose s} \left(\frac{C^2 4 s}{\pi}\right)^\frac{l-s}{2} \frac{\G(l/2)}{\G(s/2) \G(l-s+1)},\notag
\qe
 a standard Gaussian vector $g\in\R^{l}$ enjoys for all $S \subset \{1, \ldots, l\}, |S|\leq s$,
 \[
 C\|g_S\|_1\leq \|g_{S^c}\|_1.
 \]
\end{lemma}
\begin{proof}
Let $ \xi_C := \{ v \in \mathds{R}^l \text{  that does not satisfy  } NSP(s,C) \} $
and consider the joint law of standard Gaussian ordered statistics $(W_{(1)}, \dots, W_{(l)})$, then
\begin{align*}
\mathds{P}(\xi) &= 2^l l!  \int_{\mathds{R}^l}  \mathds{1}_{t \in \xi_C} \mathds{1}_{t_1 \geq \dots \geq t_l}\varphi(t_1) \dots \varphi(t_l) dt_1 \dots dt_l \\
&= \frac{2^l l!}{(l-s)!} \int_{\mathds{R}^{l-s} \times \mathds{R}^s}\mathds{1}_{t \in \xi_C} \mathds{1}_{t_1 \geq \dots \geq t_s} \varphi(t_1) \dots \varphi(t_s) \varphi(t_{s+1}) \dots \varphi(t_l) dt_1 \dots dt_l \\
&\leq \frac{2^l l!}{(l-s)!}  \left(\frac{1}{2\pi}\right)^\frac{l-s}{2} \int_{\mathds{R}^s} \mathds{1}_{t_1 \geq \dots \geq t_s} \frac{\lambda_{l-s}(B_1(C(t_1 + \dots + t_l)))}{2^{l-s}} \varphi(t_1) \dots \varphi(t_s) dt_1 \dots dt_s \,,
\end{align*}
where the last inequality relies on $\mathds{P}(\mathcal{N}(0,I_{l-s}) \in B_1(t_1 + \dots + t_s))$ is bounded by the density function of $\mathcal{N}(0,I_{l-s})$ in $0$ times the Lebesgue measure of the $l_1$ ball of radius $C(t_1 + \dots + t_s)$ in $\mathds{R}^{l-s}$. Finally, as 
\[ \lambda_{l-s}(B_1(R)) = \frac{(2R)^{l-s}}{(l-s)!}\,,\]
it implies,
\begin{align*}
\mathds{P}(\xi) &\leq \left(\frac{2 C^2}{\pi}\right)^{\frac{l-s}{2}}\frac{2^s l!}{(l-s)!^2} \int_{\mathds{R}^s} (t_1 + \dots + t_s)^{l-s}\mathds{1}_{t_1 \geq \dots \geq t_s}\varphi(t_1) \dots \varphi(t_s) dt_1 \dots dt_s \\
&= \left(\frac{2 C^2}{\pi}\right)^{\frac{l-s}{2}} \frac{ l!}{(l-s)!^2 s!} \mathds{E}\left((|W_1| + \dots + |W_s|)^{l-s}\right) \\
&= \left(\frac{2 C^2}{\pi}\right)^{\frac{l-s}{2}} {l \choose s} \frac{1}{(l-s)!} \mathds{E}( \lVert W \rVert_1^{l-s}) \,,
\end{align*}
where $W$ is a standard Gaussian vector in $\mathds{R}^s$. At last, using bound on $l_1$ norm, it comes,
\begin{align*}
\mathds{P}(\xi) &\leq \left(\frac{2 C^2}{\pi}\right)^{\frac{l-s}{2}} {l \choose s} \frac{1}{(l-s)!} s^\frac{l-s}{2} \mathds{E}( \lVert W \rVert_2^{l-s}) \\
& = \left(\frac{2 C^2}{\pi}\right)^{\frac{l-s}{2}} {l \choose s} \frac{1}{(l-s)!} s^\frac{l-s}{2} 2^\frac{l-s}{2} \frac{\G(l/2)}{\G(s/2)} \,,
\end{align*}
where the last equality follows from classical results on the moment of the $\chi$ distribution.
\end{proof}

\section{Spherical Hessian}
\begin{lemma}\label{spherical} Denote $X''(.)$ the Hessian of $X(.)$ along the sphere $\mathds{S}^{m-1}$ then
\[X''(t)  = -X(t) I_{m-1}.\]
\end{lemma}
\begin{proof}
To compute the spherical Hessian, since every point plays the same role, we can compute it at the "east pole"  $t = e_1$,  the first vector of the canonical basis.  Consider a basis $(w_2, \dots, w_n)$ of the tangent space  at $t = e_1$ and use, as a chart of the sphere,  the orthogonal projection on this space. 
 
  Let $Y (t_2, \dots,t_m)$   be the process $X(.)$  written in this chart in some neighborhood of $e_1$.  By the Pythagorean theorem,
  $$
  Y (t_2, \ldots,t_m)=X( \sqrt {1- t_2^2-\cdots-t_m^2} , t_2, \ldots, t_m).
  $$
   Plugging this into  the order two Taylor expansion of the process $X(.)$ at $t = e_1$ gives 
  \begin{align*}
  Y (t_2, \ldots,t_m) = X(e_1) + t_2 & X'_2(e_1) + \cdots +t_mX'_m(e_1) +
  \sum_{2\leq i,j\leq m} \frac{t_i t_j}{2} X''_{ij}(e_1) \\ &- X'_1(e_1)\frac{ t_2^2+\cdots +t_m^2}{2} + o(t_2^2+\cdots +t_m^2),
  \end{align*}
  where $X'_k(e_1) = \frac{\partial X}{\partial w_k}(e_1)$ and $X''_{ij}(e_1) = \frac{\partial X}{\partial w_i \partial w_j}(e_1) $. As the process is locally linear, in a small enough neighborhood of $e_1$, $X''_{ij}$ is equal to zero and, by identification,
  \[ X''(e_1) = - X'_1(e_1) I_{m-1} = - X(e_1) I_{m-1}  \]
  giving the desired result.
  \end{proof}

 \bibliographystyle{plain}
 \bibliography{biblio}

\end{document}